\documentclass[11pt,reqno]{amsart}
\usepackage{amsbsy,amsfonts,amsmath,amssymb,amscd,amsthm}
\usepackage{mathrsfs}
\usepackage{subfigure,enumitem,graphicx}
\usepackage[a4paper,text={6in,7.5in},centering,includefoot,foot=1.2in]{geometry}
\usepackage{pxfonts}
\usepackage{srcltx,color}

\small\normalsize
\theoremstyle{plain}
\newtheorem{mainthm}{Theorem}
\newtheorem{maincor}[mainthm]{Corollary}
\newtheorem{thm}{Theorem}[section]
\newtheorem{cor}[thm]{Corollary}

\newtheorem{defi}[thm]{Definition}
\theoremstyle{definition}

\newtheorem{rem}[thm]{Remark}
\newtheorem{question}{Question}

\newcommand{\eqdef}{\stackrel{\scriptscriptstyle\rm def}{=}}
\DeclareMathOperator{\IFS}{IFS}

\let\oldtocsection=\tocsection
\let\oldtocsubsection=\tocsubsection
\renewcommand{\tocsection}[2]{\hspace{0em}\bf\oldtocsection{#1}{#2}}
\renewcommand{\tocsubsection}[2]{\hspace{1em}\oldtocsubsection{#1}{#2}}

\begin{document}
\title[Ergodicity of iterated function systems]{Ergodicity of iterated function systems via minimality\\ on the hyper spaces}
\author[A. Sarizadeh]{Aliasghar Sarizadeh}
\address{\centerline{Department of Mathematics, Ilam University,}
   \centerline{Ilam, Iran.}}
\email{ali.sarizadeh@gmail.com}
\begin{abstract}
We give a sufficient condition for the ergodicity of  the Lebesgue measure for an iterated function system of diffeomorphisms.
This is done via the induced iterated function system on the space of continuum (which is called hyper-space).
We introduce a notion of minimality for induced IFSs which implies that the Lebesgue measure is ergodic for the original IFS.
Here, to beginning, the required regularity is $C^1$.
However, it is proven that the $C^1$-regularity is a redundant condition to prove ergodicity
with respect to the class of quasi-invariant measures.

As a consequence of mentioned results, we obtain ergodicity with respect to
Lebesgue measure for several systems.
\end{abstract}
\maketitle
\thispagestyle{empty}
\section{Introduction}
The word ``\emph{Ergodic}" comes from classical statistical mechanics.
In the context of dynamical systems, ergodic theory is the statistical study
of systems relative to at least quasi-measures.
Actually, ergodicity is concerned with the behavior of null and conull invariant sets.

Nowadays, the existence of a relationship between  minimality  and ergodicity can not be denied.
This relation between minimality and ergodicity yields to the following question:
Under what conditions a minimal iterated function system on
a compact differentiable manifold is volume-ergodic?
For instance, by $C^{1+\alpha}$-regularity and expanding assumptions, Navas  proved  that minimality implies Lebesgue ergodicity for a group action on the circle (Ref.  \cite{Na04}).
But the generalization result of Navas to higher-dimensional manifolds  seems to have a much more complicated
structure. In Ref. \cite{bfms}, by additional assumption of conformality  of generators,
authors obtained Lebesgue ergodicity for a group action
on higher-dimensional manifolds. We refer to Ref. \cite{S15}, for a robust example of conformal minimal systems, in dimension two.

On the other hands, in Ref. \cite{OR01}, it is known that the $C^1$-regularity is not
a sufficient condition to conclude ergodicity from minimality. More interesting cases are taken into consideration concern to the
context of $C^{1+\alpha}$-diffeomorphisms ($C^1$-diffeomorphisms with $\alpha$-H\"{o}lder derivatives
with $\alpha> 0$).

Summing up what previously mentioned: it seems that $C^{1+\alpha}$-regularity and some kinds of hyperbolicity
(e.g. uniform, non-uniform or partial) play essential role to prove ergodicity.

The aim of present work is to establish the ergodicity for some
$C^1$-regular systems which can be far from hyperbolicity, in some sense.
On the other hand, it seems that ergodicity
was never studied via the induced iterated function systems on hyper-spaces.
It implies that we study the problem  from a different point of view.
\subsection{Minimality and Ergodicity of iterated function systems}
An iterated function system generated by finitely many maps is the collection of all possible
compositions of the maps which is a popular way
to generate and explore a variety of fractals.
More precisely,
consider finitely many maps
$\mathcal{F}=\{f_{1},\dots,f_{k}\}$ on a compact metric space $(X,d)$.
Write $\langle\mathcal{F}\rangle^+$ for the semigroup generated by collection $\mathcal{F}$ under function composition.
The action of the semigroup $\langle\mathcal{F}\rangle^+$ on $X$ is called the iterated function system
associated to $\mathcal{F}$ and we denote it by $\IFS(X;\mathcal{F})$ or $\IFS(\mathcal{F})$.

We said that the $\IFS(X;\mathcal{F})$ is minimal if every invariant non-empty closed subset of $X$
is the whole space $X$, where $A\subset X$ is invariant for $\IFS(X;\mathcal{F})$ if
$$
           f(A)\subseteq A;\ \ \ \ \forall f\in \mathcal{F}.
$$
In additional, let $X$ equipped to a quasi-invariant probability measure $\mu$.
The $\IFS(X;\mathcal{F})$ is
\emph{ergodic} with respect to $\mu$  if $\mu(A)\in \{0,1\}$ for all
invariant set $A$ of $\IFS(X;\mathcal{F})$.

Observe that the counterpart to ergodicity is minimality, in topological point of view.
So, it is logical that the relation between Ergodicity and minimality have been studied extensively by many authors;
(see for instance see Refs. \cite{bfms}, \cite{dkn}, \cite{f61}, \cite{OR01} and \cite{S15}).

As previously mentioned, we will focus on the Lebesgue measure which is
quasi-invariant for $C^1$ (local) diffeomorphisms.

\subsection{Induced maps on hyper-spaces}
Let $f : X\to X$ be a continuous map from a continuum (that is, a compact connected
metric space) $X$ into itself. Consider the hyper-space
$$
\mathcal{K}(X)=\{A\subseteq X:\ A~\mathrm{is~non-empty~compact~connected~subset~of~} X\}
$$
with the Hausdorff metric $d_H$  on it which denoted by $(\mathcal{K}(X), d_H)$.
It is known that if $X$ is compact,
so is the hyperspace $\mathcal{K}(X)$ (see Ref. \cite{n}). A induced map is defined as follow
$$
  \hat{f}: \mathcal{K}(X) \to \mathcal{K}(X), \quad  \hat{f}(A)\eqdef f(A)=\bigcup_{a\in A}\{f(a)\}.
$$
Consider a classical dynamical system $(J,f)$
where $f$ is continuous and $J$ is  one-dimensional continuum.
Although, the systems $(J,f)$ and  $(\mathcal{K}(J),\hat{f})$ have a similar behavior in
some phenomena (see Refs. \cite{f}, \cite{frs} and \cite{m}),  but there exist some other phenomena in
which  the behavior of these systems are  completely different. More precisely,
the induced map is never transitive, even if $f$ happens to be Refs. \cite{aim} and \cite{ko}.
Clearly, the induced map is never minimal, too.

The main goal of this note is study of iterated function systems generated by finite induced maps to getting ergodicity of original system.
\subsection{Induced IFSs and ergodicity of orginal IFSs}
Now, consider $\IFS(X;\mathcal{F})$ with
$\mathcal{F}=\{f_{1},\dots,f_{k}\}$.
Naturally, we can define the induced iterated function system
$\IFS(\mathcal{K}(X);\hat{\mathcal{F}})$ on
the space $(\mathcal{K}(X), d_H)$ in which $\hat{\mathcal{F}}=\{\hat{f_1},\dots,\hat{f_k}\}$ and
$\hat{f}_i$ is induced map, for  $1\leq i\leq k$.

As previously mentioned, induced map is never minimal or even transitive.
So, logically,  the following question arises.
\begin{question}
Is there $\IFS(X;\mathcal{F})$ so that $\IFS(\mathcal{K}(X);\hat{\mathcal{F}})$ is minimal or transitive?
\end{question}
In this note, independent of the answer of above question, we define a weak kind of minimality for
induced iterated function systems on hyper-space.
\begin{defi}\label{def1}
For $\Theta>5$, we say that  $\IFS(X;\mathcal{F})$ is a
$\Theta$-hyper-minimal if for some $r_0>0$ and every $x,y\in X$ the following hold:
for every every $0<r<r_0$ there exists $\hat{h}=\hat{h}(r,x,y)\in \IFS(\mathcal{K}(X);\hat{\mathcal{F}})$ so that
$$
          d_H(\hat{h}(B(x,r)^\prime),B(y,r)^\prime)<r/\Theta,
$$
where $B(x,r)$ is the geodesic ball of radius $r$ centered about $x$ and
$B(x,r)^\prime$ is the set of all limit points of the set $B(x,r)$.
\end{defi}
In many works to prove ergodicity, bounded distortion is a main ingredient in the proof.
An important things about $\Theta$-hyper-minimal is that bounded
distortion is inherent in the definition of $\Theta$-hyper-minimal.
Then there is no reason  to restrict our generators to especial maps
(contracting or expanding maps), when we do not have to apply
the classical bounded distortion (compare with Refs. \cite{bfms} and \cite{S15}).

Now, we are ready to formulate the main result of this paper.
\begin{mainthm}\label{mainA}
Suppose that $M$ is a smooth compact differential
manifold and the $\IFS(M;\mathcal{F})$ is $\Theta$-hyper-minimal,  for
$\mathcal{F}=\{f_1,\dots,f_k\}\subset \mathrm{Diff}^{1}(M)$. Then the system $\IFS(M;\mathcal{F})$ is ergodic
with respect to Lebesgue measure.
\end{mainthm}
The next result states a generalized form of Theorem \ref{mainA}.
\begin{maincor}\label{mainB}
Every $\Theta$-hyper-minimal IFS generated by finitely many homeomorphisms, which the
Lebesgue measure is quasi-invariant\footnote{A measure
$\mu$ is said to be quasi-invariant if $(f_i)_\ast\mu$ is absolutely
continuous with respect to $\mu$ for every $i=1,\dots k$.} for them,
is ergodic with respect to Lebesgue measure.
\end{maincor}
As another consequence of Theorem \ref{mainA}, we obtain ergodicity for some generic ordinary systems.
Let
$$
            \mathcal{CCT}_m(\mathbb{T}^2)=\{hR_\alpha h^{-1}:\ h\in \text{Home}_m(\mathbb{T}^2),\ \alpha\in \mathbb{T}^2\},
$$
where $\text{Home}_m(\mathbb{T}^2)$ consists of all homeomorphisms that the
Lebesgue measure $m$  is quasi-invariant with respect to them and
$R_\alpha$ for $\alpha=(\alpha_1,\alpha_2)$ is the rigid translation
$$
         R_\alpha=R_{(\alpha_1,\alpha_2)}: \mathbb{T}^2\to \mathbb{T}^2;\ \ \ (x,y)\mapsto (x+\alpha_1,y+\alpha_2).
$$

We denote  the closure of the set $\mathcal{CCT}_m(\mathbb{T}^2)$ in the $C^0$-topology by
$\overline{\mathcal{CCT}_m}^{_{0}}(\mathbb{T}^2)$.
\begin{maincor}\label{mainC}
Ergodicity with respect to Lebesgue measure in the space $\overline{\mathcal{CCT}_m}^{_{0}}(\mathbb{T}^2)$
is a generic property.
\end{maincor}
We remark that the metric entropy of every elements of $\overline{\mathcal{CCT}_m}^{_{0}}(\mathbb{T}^2)$ is zero.
\section{Proof of main result: theorem \ref{mainA}}
In this section,  we prove Theorem \ref{mainA} and also, we state and prove another form of it.
To this end, we begin with the following definition.
\begin{defi}
Consider two different positive numbers $0\leq t\leq 1/2,\ 0< \ell<1$ and also, consider $\Theta>1$.
We say that $\Theta$ is a $(t,\ell)$-overlap number on a smooth compact differential manifold $M$ if for every
$r>0$ and $x\in M$ the following hold:
\begin{equation}\label{eq1}
m(B(x,r)\bigcap B(y,r-r/\Theta))-tm(B(x,r))>\ell m(B(x,r-r/\Theta));\ \ \forall y\in B(x,r/\Theta),
\end{equation}
where $m$ is normalized Lebesgue measure
\end{defi}
Observe that if $1/\Theta$ are necessary  close to zero,
one can ensure that always exist $t, \ell$ in which
$t, 1-\ell$ are sufficiently  closed to zero and
$\Theta$ is a $(t,\ell)$-overlap number. In fact, in Definition \ref{def1},
we put $\Theta>5$  for existence suitable
numbers $t,\ell$ so that $\Theta$ is a $(t,\ell)$-overlap number.
Thus, hereafter, we take $\Theta>5$.

\subsection{Global dynamics}
Before we start to prove Theorem \ref{mainA}, notice that if the
$\IFS(M;\mathcal{F})$ is $\Theta$-hyper-minimal, then the system $\IFS(M;\mathcal{F})$ is minimal.
Indeed, let $x,y$ be two arbitrary points in $M$ and $B(y,a)$ be a neighborhood of $y$.
By the assumption, one can find $r<a/2$ and $\hat{h}\in \IFS(\mathcal{K}(M);\hat{\mathcal{F}})$ so that
$d_H(\hat{h}(B(x,r)^\prime),B(y,r)^\prime)<r/\Theta<r/4$, for some $\Theta>5$ .
Thus $\hat{h}(B(x,r))\subset B(y,a)$, which implies that minimality of $\IFS(M;\mathcal{F})$.
\begin{proof}[Proof of Theorem~\ref{mainA}]
Suppose that $M$ is a smooth compact differential
manifold and $m$ is normalized Lebesgue measure.
Also, assume that $\IFS(M;\mathcal{F})$ is $\Theta$-hyper-minimal.
Since $\Theta>5$, there are $t,\ell$ so that $\Theta$ is a $(t,\ell)$-overlap number.
Suppose that $0<m(\mathcal{B})<1$  and
$f_i(\mathcal{B})\subset\mathcal{B}$ for all $i=1,\dots,k$. It is implies that $h(\mathcal{B})\subset \mathcal{B},$ for every
$h\in\IFS(M;\mathcal{F})$. So, one can have
$h(\mathcal{B}\bigcap B(p,r))\subseteq \mathcal{B}\bigcap h(B(p,r))$.

One can choose a point $p\in DP(\mathcal{B})$ when  $0<m(\mathcal{B})=m(\mathrm{DP}(\mathcal{B}))$.
Hence there is $\kappa_0>0$ so that for every $0<\kappa<\kappa_0$
$$
 m(\mathcal{B}:B(p,\kappa))=\frac{m(\mathcal{B}\cap
 B(p,\kappa))}{m(B(p,\kappa))}\geq 1-t.
$$
Since $\IFS(M;\mathcal{F})$ is $\Theta$-hyper-minimal, by definition, for some $0<r_0<k_0$ and every $x\in M$ the following hold:
for every $0<r<r_0$ there exists $\hat{h}=\hat{h}(r,p,x)\in \IFS(\mathcal{K}(M);\hat{\mathcal{F}})$ so that
$$
          d_H(\hat{h}(B(p,r)^\prime),B(x,r)^\prime)<r/\Theta,
$$
and it implies that $\hat{h}(B(p,r))\subseteq B(x,r+r/\Theta)$.
In particular, it holds for  $x\in B(p,r/\Theta)$.
On the other hand, $\Theta$ is a $(t,\ell)$-overlap number. Thus,
$$
 m(B(p,r)\bigcap B(x,r-r/\Theta))-tm(B(p,r))>\ell m(B(p,r-r/\Theta)),
$$
for  $x\in B(p,r/\Theta)$. It follows that
$$
\ell m(B(p,r-r/\Theta))<m(\hat{h}(B(p,r))\bigcap\mathcal{B})<m( B(x,r+r/\Theta)).
$$
Take $J=B(x,r)$. Hence, one can have
\begin{align*}
\frac{m(\mathcal{B}\bigcap J)}{m(J)}&\geq
\frac{\frac{\Theta}{\Theta+1}m(\mathcal{B}\bigcap \hat{h}(B(p,r)))}{m(J)}\\
&\geq \frac{\Theta}{\Theta+1}\cdot\frac{m(\mathcal{B}\bigcap \hat{h}(B(p,r)))}{m(J)}\\
&> \frac{\Theta}{\Theta+1}\cdot\frac{\ell m(B(p,r-r/\Theta))}{m(J)}\\
&> \frac{\Theta}{\Theta+1}\cdot\ell\cdot \frac{ m(B(p,r-r/\Theta))}{m(J)}\\
&>\frac{\Theta}{\Theta+1}\cdot\ell\cdot c
\end{align*}
where the constant $c$ is equal to $\frac{ m(B(p,r-r/\Theta))}{m(J)}$.
This  means that $\frac{m(\mathcal{B}\bigcap J)}{m(J)}$ is
bounded from below for every neighborhood $J$ of $x$.
It is equivalent to: for every $x\in M$, $x\not\in DP(\mathcal{B}^c)$ which is a contradiction, when
$\IFS(M;\mathcal{F})$ is minimal and the Lebesgue measure is a quasi-invariant for any of generators.
\end{proof}
\begin{rem}
In the proof of Theorem \ref{mainA}, we used the $C^1$-regularity just to show that the Lebesgue measure is quasi-invarint, so the
 proof of Corollary~\ref{mainB} is similar to Theorem \ref{mainA}.
\end{rem}
\subsection{Local dynamics}
Its visible from the proof of Theorem \ref{mainA}, the definition \ref{def1} has more than which we need to proof ergodicity.
So, one can define a locally version of its.
\begin{defi}\label{de1}
We say that  $\IFS(M;\mathcal{F})$ is locally $\Theta$-hyper-minimal if there exists an open set $U\subseteq M$, measurable subset $U^\prime$ of $U$
with $m(U^\prime)>0$ and $r_0>0$ so that for all $x\in U^\prime$ and $y\in U$ in which the following holds:
for every every $0<r<r_0$, there exists
$\hat{h}=\hat{h}(r)\in \IFS(\mathcal{K}(M);\hat{\mathcal{F}})$ so that
$$
          d_H(\hat{h}(B(x,r)^\prime),B(y,r)^\prime)<r/\Theta.
$$
\end{defi}
\begin{cor}\label{local}
Suppose that $M$ is a smooth compact differential
manifold and $\IFS(M;\mathcal{F})$ is minimal and locally $\Theta$-hyper-minimal, for
$\mathcal{F}=\{f_1,\dots,f_k\}\subset \mathrm{Diff}^{1}(M)$.
Then the system $\IFS(M;\mathcal{F})$ is ergodic with respect to Lebesgue measure.
\end{cor}
\begin{proof}
Suppose that $0<m(\mathcal{B})<1$  and $f_i(\mathcal{B})\subset\mathcal{B}$ for all $i=1,\dots,k$.
Also, suppose that $U\subseteq M$ is an open set and  $U^\prime$ is a measurable subset $U$ with $m(U^\prime)>0$
so that they satisfy in Definition \ref{de1}. On the other hand, $m(DP(B^c))+m(DP(B))=1$ and $0<m(\mathcal{B})m(\mathcal{B}^c)<1$.
Thus, $U^\prime \cap DP(B^c)\neq \emptyset$ or $U^\prime \cap DP(B)\neq \emptyset$.
Without loss of generality, let $p\in U^\prime \cap DP(B)\neq \emptyset$.
Hence there is $\kappa_0>0$ so that for every $0<\kappa<\kappa_0$
$$
 m(\mathcal{B}:B(p,\kappa))=\frac{m(\mathcal{B}\cap
 B(p,\kappa))}{m(B(p,\kappa))}\geq 1-t.
$$
Let $y$ arbitrary point $ U$, by assumption, for every $r<r_0$ with $r_0<k_0$ there exists
$\hat{h}\in \IFS(\mathcal{K}(M);\hat{\mathcal{F}})$ so that
$$
              d_H(\hat{h}(B(p,r)^\prime),B(y,r)^\prime)<r/\Theta.
$$
By similar argument used in the proof of Theorem \ref{mainA}, for  $r<r_0$ one can prove that
 $\frac{m(\mathcal{B}\bigcap B(y,r))}{m(B(y,r))}$ is
bounded from below for every $y\in U$.
It is equivalent to: for every $y\in U$, $y\not\in DP(\mathcal{B}^c)$. Then $m(U^\prime \cap DP(B))=m(U)$
which is a contradiction, when  $\IFS(M;\mathcal{F})$ is minimal and the
Lebesgue measure is a quasi-invariant for any of generators.
\end{proof}
\section{Examples}
We are going to consider several examples, which all of them are ergodic with respect to Lebesgue measure.
\subsection{Example on the circle}
In this example, we construct an iterated function systems on
hyper-space over the circle. To this end, let 
$I_1,I_2$ be two open connected subsets of $S^1$ with
the following property
\begin{enumerate}[label=(\roman*),ref=\roman*]
\item\label{it:0} $I_1\bigcup I_2=S^1$,
\item\label{it:1} $m(S^1\setminus I_1)<\frac{1}{20}$,
\item\label{it:2} for all $x\in S^1\setminus I_1,\ B(x,1/4)\subset I_2$,
\end{enumerate}
Let $f_1, f_2$ be two $C^{1+\alpha}$ function on the circle so that
 $$
 f_1|_{I_1}=R_\beta,\ f_2|_{I_2}=R_\gamma
 $$
where  $\beta\thickapprox\gamma\in \mathbb{Q}^c$,  and $\beta\gtrapprox m(S^1\setminus I_1)$.
Take $\mathcal{F}=\{f_1,f_2\}$.
We claim that  $\IFS(\mathcal{K}(S^1);\hat{\mathcal{F}})$ is $\Theta$-hyper-minimal.
Since we work on one dimension, it means that for every $x,y\in S^1$ and some  $r_0>0$ the following holds: for every $r<r_0$ there exists
$\hat{h}\in \langle\hat{\mathcal{F}}\rangle^+$ in which
$$
d_H(\hat{h}(B(x,r)),B(y,r))<r/5.
$$
Since $x\in I_1$ (resp. $x\in I_2$), one can apply  $f_1$ (resp. $f_2$) and so, from this point of view,
$f_\omega(x)=R_\omega(x)$
for some $\omega\in \Sigma^+_k$.

On the other hand, from \cite{bfs}, if $\omega$ is a dense sequence under the shift map then the fiberwise orbit dive by it is
dense on the circle for all $x\in S^1$. However,
from \cite{k}, for every $x\in S^1$, there exists $\omega\in \Sigma^+_k$
 such that $\mathcal{O}^+_\omega(x)$ is nowhere dense on the circle.

Take
$$
    k=max\{n\in \mathbb{N};\ n\beta\leq1-\gamma\}\in\mathbb{Q}.
$$
Now, for $x$ in $X$, without loss of generality, let $x\in I_1$. We apply $f_1$ on $x$. Again, if $f_1(x)\in I_1$ and
$S^1\setminus I_1 \nsubseteqq [x,f_1(x)]$ apply $f_1$ on $f_1(x)$.
Otherwise, if $f_1(x)\not\in I_1$ or
$S^1\setminus I_1 \subseteqq [x,f_1(x)]$ apply $f_2$ on $f_1(x)$. By an inductive process,
one can construct $\omega\in\Sigma^+_k$
so that
$$
\lim_{n\to \infty}\frac{1}{n}\text{Card}\{i\in \mathbb{N}:\ \omega_i=1\ \text{and}\ i\leq n\}=\frac{k-1}{k},
$$
which implies that $\overline{\mathcal{O}_\omega^+(x)}=X$.
On the other hand, by construction $\omega$, for every $i\in\mathbb{N}$, $f_\omega^i(x)=R_\omega^i(x)$.
Therefore, there is $i\in \mathbb{N}$ so that $f_\omega^i(x)$ is sufficiently close to $y$.
Also, when $f_\omega^i(x)=R_\omega^i(x)$
by choosing suitable $i$, one can have $d_H(B(y,r), f_\omega^i(B(x,r)))<r/5$ for every $r$
less that the Lebesgue number of the open covering $\{I_1,I_2\}$.

Hence $\IFS(S^1;f_1,f_2)$ is ergodic with respect to Lebesgue measure.
\subsection{Example on the torus}\label{sec}
Now, we prove that the Lebesgue measure is ergodic for some conjugancy class of translations
of $\mathbb{T}^2$. For $x\in M$, let a subset $\Gamma_x$ of  $\mathrm{Home}(M)$ be defined as follow:
the set of all $g\in \mathrm{Home}(M)$ so that there exist
an invertible  matrix $A_{2\times 2}$ and  a neighborhood $U_x$ of $x$ with
$$
g(z)=A(z-x)+g(x);\ \ \ \forall z\in U_x
$$
Now, let $\gamma\in \mathbb{Q}^c\times\mathbb{Q}^c\bigcap\mathbb{T}^2$ and $h\in\Gamma_{x_0}$,
for some $x_0\in \mathbb{T}^2$.
Since $R_\gamma$ is minimal, the conjugate of it $hR_\gamma h^{-1}$ is minimal, too.
Thus, it is sufficient to prove that  $\IFS(\mathcal{K}(\mathbb{T}^2);\widehat{hR_{\gamma}h^{-1}})$ is locally $\Theta$-hyper-minimal,
for a $\Theta$ greater than $5$.

Since $h\in \Gamma_{x_0}$, one can fined $0<\rho_h$ so that
$Dh(z)=Dh(y)=A$  for every $y,z\in B(x_0,\rho_h)$ and $Dh^{-1}(z)=Dh^{-1}(y)=A^{-1}$ for every $y,z\in h(B(x_0,\rho_h))$.
Take $r_h=\rho_h/2$ and $y\in B(x_0,r_h)$.
For $0<r<r_h$ and using of construction of $h,h^{-1}$, it is not hard to find that there is $n_j$ so that for $z,y \in B(x_0,r_h)$
$$
     d_H(hR^{n_j}_\gamma h^{-1}(B(z,r)^\prime),B(y,r)^\prime)<r/\Theta.
$$
Summing up, the following holds:
\begin{cor}\label{it:hddd}
Let $\gamma\in \mathbb{Q}^c\times\mathbb{Q}^c\bigcap\mathbb{T}^2$ and $h\in\Gamma_{x_0}$,
for some $x_0\in \mathbb{T}^2$.
Then $\IFS(\mathbb{T}^2;hR_\gamma h^{-1})$ is ergodic with respect to Lebesgue measure.
\end{cor}
\subsection{Example on the sphere}
Here, a sphere $S^2$ is defined as the set of points that are all at the
same distance $r=\frac{1}{2\pi}$ from $0=(0,0,0)$, in three-dimensional space.

Take a particular point $P=(0,0,r)$ on a sphere as its north pole,
then the corresponding antipodal point $Q=(0,0,-r)$ is south pole.
Notice that the equator $S^1$ of sphere, is great circle which lies in $xy$-space with length one.
We equip $S^1$ to an orientation and we will denote by $\stackrel{\frown}{ab}$ the arc from $a$ to $b$
according to this orientation.
Also, the circle of sphere which through the point
$H$ and parallel to equator $S^1$, denoted by $\Upsilon_{S^1}(H)$.

Observe that one can define for $\gamma\in (0,1)$ a
preserving orientation map $\Theta_\gamma:S^1\to S^1$
as follows $\Theta_\gamma(a)=b$,  where the length of
the arc $\stackrel{\frown}{ab}$ is equal to $\gamma$.
It is well known that $\Theta_\gamma$ is minimal if $\gamma$ is irrational.

Also, let  $H=(x_1,y_1,z_1)$ be a point on $2$-sphere.
Take $\Xi(P,Q:H)$ the meridian longitude which contains points
$H,P,Q$ and also take $\Gamma(P,Q:H)$ the connected part
of $\Xi(P,Q:H)\setminus \{P,Q\}$ which contains point $H$.
Define a map $\pi_{\{P,Q\}}:S^2\setminus\{P,Q\}\to S^1$ as follows
$\pi_{\{P,Q\}}(H)=\Gamma(P,Q:H)\bigcap S^1$.
Observe that the map $\pi_{\{P,Q\}}$ is not invertible but correspond to every point $K$ in $S^1$,
one can correspond to  $K$, $\psi_{\{P,Q\}}^H(K)$ on $\Upsilon_{S^1}(H)$ so that
$\psi_{\{P,Q\}}^{H}(K)=\Upsilon_{S^1}(H)\bigcap \Gamma(P,Q:K)$ and
$$
      \psi_{\{P,Q\}}^{H}(\pi_{\{P,Q\}}(H))=H.
$$
The translations $T_{\gamma}:S^2\to S^2$ have the form
$$
T_\gamma(H)=
\left\{
  \begin{array}{ll}
    T_\gamma(H)=H, & \hbox{H=p\ \text{or}\ q;} \\
    T_{\gamma}(H)=\psi_{\{P,Q\}}^{H}(\Theta_\gamma(\pi_{\{P,Q\}}(H))), & \hbox{otherwise}.
  \end{array}
\right.
$$
Now,  take a particular point $e=(r,0,0)$ on a sphere as its north pole,
then the corresponding antipodal point $w=(-r,0,0)$ is south pole.
For this additive notation, take $\hat{S^1}$ the equator of
sphere which is great circle which lies in $yz$-space with length one.
Again, we will equip $\hat{S^1}$ to an orientation and we will define
$\hat{\Theta}_\gamma:\hat{S^1}\to \hat{S^1}$
as follows $\hat{\Theta}_\gamma(x)=y$,  where the length of the arc from $x$ to $y$
according to its orientation is $\gamma$.

The translations $R_{\gamma}:S^2\to S^2$ have the form
$$
R_\gamma(H)=\left\{
  \begin{array}{ll}
    R_\gamma(H)=H, & \hbox{H=e\ \text{or}\ w;} \\
    R_{\gamma}(H)=\psi_{\{e,w\}}^{H}(\hat{\Theta}_\gamma(\pi_{\{e,w\}}(H))),& \hbox{otherwise.}
  \end{array}
\right.
$$
Now, consider the IFS generated by
$\mathscr{F}=\{f_1, f_2\}$, where $f_1=T_{\gamma_1}$ and $f_2=R_{\gamma_2}$ for irrational numbers $\gamma_1,\gamma_2$.
It is not hard to show that the following hold
\begin{enumerate}
\item\label{it:1} Both of $\IFS(S^2;\mathscr{F})$ and $\IFS(S^2;\mathscr{F}^{-1})$ are minimal.
\item\label{it:2} Both of $\IFS(S^2;\mathscr{F})$ and $\IFS(S^2;\mathscr{F}^{-1})$ is equicontinuous.
\item\label{it:3} $\langle\mathscr{F}\bigcup\mathscr{F}^{-1}\rangle^+$ does not contain any minimal element.
\item\label{it:4} The Lebesgue measure is ergodic for both of $\IFS(S^2;\mathscr{F})$ and $\IFS(S^2;\mathscr{F}^{-1})$.
\end{enumerate}
\begin{rem}
Notice that both $\IFS(S^2;\mathscr{F})$ and $\IFS(S^2;\mathscr{F}^{-1})$
satisfies the deterministic chaos game by Theorem 3 in \cite{L15}.
\end{rem}
\section{Proof of Corollary~\ref{mainC}: Generic ergodicity of some systems}
Recall that a subset $\mathcal{R}$ is residual if it
contains a countable intersection of open dense sets.
The  space $C^0(\mathbb{T}^2)$ with the $C^0$-topology is a Polish space, so that
any closed subset of it (in particular, $\overline{\mathcal{CCT}_m}^{_{0}}(\mathbb{T}^2)$) is a Baire space. A property
is said to be generic in a Baire space if it holds in a residual subset of the space.

Let us that we follow the notation of the Subsection \ref{sec}.
Observe that the set $\Gamma_x$
is dense in the metric space $(\text{Home}(\mathbb{T}^2), d_0)$, for every $x\in \mathbb{T}^2$ and
$\{R_\beta:\ \beta\in \mathbb{Q}^c\times\mathbb{Q}^c\}$ is dense in the set of all translations
with the metric $d_0$, then by  Corollary \ref{it:hddd} the following holds:
the metric space $(\overline{\mathcal{CCT}_m}^{_{0}}(\mathbb{T}^2),d_0)$ contains of a dense subset
which every element of it is  ergodic with respect to the Lebesgue measure.

Moreover, for $h\in \Gamma_x$, we define $\mathcal{U}_n(h)$
consist of all $fR_\gamma f^{-1}$ for some $\gamma\in\mathbb{T}^2$ with the following properties
\begin{enumerate}
\item $f\in \text{Home}(\mathbb{T}^2)$ and $\|f-h\|_1<1/n$,
\item $d_H(fR^{j}_\gamma f^{-1}(B(z,r)^\prime),B(y,r)^\prime)<r/\Theta$; for $r_h<r<\frac{r_h}{n+1},\ z,y\in B(x,r_h)$  and some $j\in \mathbb{N}$,
\end{enumerate}
Since $\mathcal{U}_n(h)$ is a non-empty open set and  $\Gamma_x$
is dense in the space $(\text{Home}(\mathbb{T}^2), d_0)$, the set $\mathcal{U}_n=\cup_{h\in\Gamma_x} \mathcal{U}_n(h)$
is an open and dense subset of $(\overline{\mathcal{CCT}_m}^{_{0}}(\mathbb{T}^2),d_0)$.
Take $\mathcal{U}=\cap_n\mathcal{U}_n$.
It is clear that every $f\in \mathcal{U}$, is locally
$\Theta$-hyper-minimal, for a $\Theta$ greater than 5.

\textbf{Acknowledgments.}
We thank P. G. Barrientos, A. Fakhari and M. Nassiri for useful discussions
and suggestions.
\def\cprime{$'$}

\end{document}